\documentclass[numbers,webpdf,imanum]{ima-authoring-template}%

\graphicspath{{Fig/}}

\theoremstyle{thmstyletwo}%
\newtheorem{theorem}{Theorem}
\newtheorem{lemma}{Lemma}

\theoremstyle{thmstylethree}%

\newtheorem{assumptions}{Assumptions}

\newtheorem{example}{Example}%
\newtheorem{remark}{Remark}%

\numberwithin{equation}{section}

\usepackage{bm}

\def\dd{\hspace*{1pt}\mathrm{d}}
\def\ee{\mathrm e}

\def\Co_#1{C}
\def\Co_#1{C_{#1}}

\def\NN{\mathbb N}

\def\RR{\mathbb R}
\def\CC{\mathbb C}

\def\Ell{\mathrm{L}}

\def\Lip{\mathrm{Lip}}
\def\Ell{\mathrm{L}}
\def\He{\mathrm{H}}

\def\sup{\operatornamewithlimits{sup}}

\def\max{\operatornamewithlimits{max}}

\newcommand{\downto}{\downarrow}

\def\tmax{t_{\mathrm{max}}}

\newcommand{\TT}{\mathbb{T}}
\def\st{:}

\def\LIP{\ell_{\Sigma}}
\def\LIPB{\ell_{\Sigma,B}}

\newcommand{\calT}{\mathcal{T}}
\newcommand{\calA}{\mathcal{A}}

\newcommand{\calO}{\mathcal{O}}
\newcommand{\calR}{\mathcal{R}}

\newcommand{\calF}{\mathcal{F}}

\newcommand{\calTT}{\bm{T}}
 
\newcommand{\bfL}{\bm{L}}

\newcommand{\vecu}{\bm u}
\newcommand{\vecv}{\bm v}

\newcommand{\dom}{\mathrm{dom}}

\renewcommand\max{\operatornamewithlimits{max}}

\newcommand{\vect}[2]{\tbinom{#1}{#2}}

\newcommand{\eloc}{\ee^{\textrm{loc}}}
\def\str{\Sigma}

\makeatletter
\newenvironment{iiv}{\begin{enumerate}[9999]}{\end{enumerate}}

\newenvironment{abc}{\begin{enumerate}[9]}{\end{enumerate}}

\newenvironment{numer}{\begin{enumerate}[99]}{\end{enumerate}}
\newenvironment{numerA}{\begin{enumerate}[99]}{\end{enumerate}}

\makeatother
\newenvironment{enumsteps}{%
	\begin{enumerate}[{\it Step 1.}]}{%
	\end{enumerate}
}  

\newcommand{\bfA}{\mathbf{A}}
\newcommand{\bfM}{\mathbf{M}}
\newcommand{\bfF}{\bm{\calF}}
\newcommand{\bfrho}{\bm{\varrho}}
\newcommand{\bfu}{\mathbf{u}}
\newcommand{\bfv}{\mathbf{v}}

\newcommand{\ga}{\gamma}
\renewcommand{\Omega}{\varOmega}
\renewcommand{\Gamma}{\varGamma}
\newcommand{\Ga}{\Gamma}
\newcommand{\Om}{\Omega}
\newcommand{\pa}{\partial}
\newcommand{\R}{\RR}

\newcommand{\nnu}{\textnormal{n}}

\definecolor{grey}{rgb}{0.42,0.42,0.42}
\def\bafahide#1{}
\def\pecshide#1{}

\definecolor{anormalgreen}{rgb}{0,0.72,0}

\newcommand{\blueon}{\color{black}}
\newcommand{\blueoff}{\color{black}}

\begin{document}

\DOI{DOI HERE}
\copyrightyear{2022}
\vol{00}
\pubyear{2022}
\access{Advance Access Publication Date: Day Month Year}
\appnotes{Paper}
\copyrightstatement{Published by Oxford University Press on behalf of the Institute of Mathematics and its Applications. All rights reserved.}
\firstpage{1}


\title[Splitting for semilinear boundary coupled systems]{Error estimates for a splitting integrator for abstract semilinear boundary coupled systems}

\author{Petra Csom\'os
\address{\orgname{Alfr\'ed R\'enyi Institute of Mathematics}, \orgaddress{\street{Re\'altanoda utca 13--15.}, \postcode{1053}, \state{Budapest}, \country{Hungary}}}
\address{\orgdiv{Department of Applied Analysis and Computational Mathematics, and MTA-ELTE Numerical Analysis and Large Networks Research Group}, \orgname{E\"otv\"os Lor\'and University}, \orgaddress{\street{P\'azm\'any P\'eter s\'et\'any 1/C}, \postcode{1117}, \state{Budapest}, \country{Hungary}}}}
\author{B\'alint Farkas
\address{\orgdiv{School of Mathematics and Natural Sciences}, \orgname{University of Wuppertal}, \orgaddress{\street{Gaußstraße 20}, \postcode{42119}, \state{Wuppertal}, \country{Germany}}}}
\author{Bal\'azs Kov\'acs*
\address{\orgdiv{Department of Mathematics}, \orgname{Technical University of Munich}, \postcode{85748}, \state{Garching bei M\"unchen}, \country{Germany}}
\address{\orgdiv{Faculty of Mathematics}, \orgname{University of Regensburg}, \orgaddress{\street{Universit\"atsstr. 31}, \postcode{93040}, \state{Regensburg}, \country{Germany}}}}

\authormark{P. Csom\'os, B. Farkas, and B. Kov\'acs}

\corresp[*]{Corresponding author: \href{email:balazs.kovacs@mathematik.uni-regensburg.de}{balazs.kovacs@mathematik.uni-regensburg.de}}

\received{26}{04}{2022}
\revised{Date}{0}{Year}
\accepted{Date}{0}{Year}


\abstract{We derive a numerical method, based on operator splitting, to abstract parabolic semilinear boundary coupled systems. The method decouples the linear components which describe the coupling and the dynamics in the abstract bulk- and surface-spaces, and treats the nonlinear terms similarly to an exponential integrator. 	The convergence proof is based on estimates for a recursive formulation of the error, using the parabolic smoothing property of analytic semigroups and a careful comparison of the exact and approximate flows.
This analysis also requires a deep understanding of the effects of the Dirichlet operator (the abstract version of the harmonic extension operator) which is essential for the stable coupling in our method. 
Numerical experiments, including problems with dynamic boundary conditions, reporting on convergence rates are presented.}

\maketitle

\section{Introduction}
In this paper we derive a Lie-type splitting integrator for abstract \emph{semilinear} boundary coupled systems, and prove first order error estimates for the time integrator by extending the results of \cite{CsEF} from the linear case. The main idea of our algorithm is to decouple the two nonlinear problems appearing in the original coupled system, while maintaining stability of the boundary coupling with the help of the abstract Dirichlet operator. 
We use techniques from operator semigroup theory to prove the first-order convergence in the following abstract setting.

\medskip
We consider the abstract semilinear boundary coupled systems of the form:
\begin{equation}\label{eq:main0}
\begin{cases}
\begin{aligned}
\dot u(t)&=A_m u(t) + \calF_1(u(t),v(t))&&\text{ for } 0<t\leq t_{\max},\quad u(0)=u_0\in E,\\
\dot v(t)&=B v(t) + \calF_2(u(t),v(t))&&\text{ for } 0<t\leq t_{\max},\quad v(0)=v_0\in F, \\
Lu(t)&=v(t) &&\text{ for } 0\leq t \leq t_{\max}, &
\end{aligned}
\end{cases}
\end{equation}
where $A_m$, $B$ are linear operators on the Banach spaces $E$ and $F$, respectively, $\calF_1$, $\calF_2$ are suitable functions, and the two unknown functions $u$ and $v$ are related via the linear coupling operator $L$ acting between (subspaces of) $E$ and $F$. The subscript $m$ in $A_m$ refers to the fact that this operator is maximal in a certain sense, this and the bulk--surface terminology is further explained in Example~\ref{examp:Lip}. A typical setting would be that $L\colon E \to F$ is a \emph{trace-type operator} between the  space $E$ (for the bulk dynamics) and the \emph{boundary} space $F$ (for the surface dynamics). The precise setting and assumptions for \eqref{eq:main0} will be described below. \medskip

This abstract framework simultaneously includes problems which have been analysed on their own as well.
For instance, abstract \emph{boundary feedback} systems, see   \cite{DeschLasieckaSchappacher}, \cite{DeschMilotaSchappacher}, \cite{CENN} and the references therein,  fit into the above abstract framework where the equations in $E$ and $F$ representing the bulk and boundary equations. Such examples arise, for instance, for the boundary control of partial differential equation systems, see \cite{LasieckaTriggianiI,LasieckaTriggianiII}, and \cite{KumpfNickel}, \cite[Section~3]{Engel_etal_2010}, and \cite[Section~3]{Adler_etal_2017}. These problems usually involve a bounded feedback operator acting on $u$, which can be easily incorporated into the nonlinear term $\calF_2$ above. For regularity results we refer to these works.
We further note, that semilinear parabolic equations with \emph{dynamic boundary conditions}, see \cite{Wen59,engel2005analyticity,GalG08,ColF,vazquez2011heat,Lie13,Racke2003cahn,Gal07,dynbc}, etc., and diffusion processes on \emph{networks} with boundary conditions satisfying ordinary differential equations in the vertices, see \cite{Mugnolo,MugnoloDiss,Sikolya_flows,MugnoloRomanelli,Mugnolo_book}, etc., both formally fit into this setting. In both cases, however, the feedback operator is unbounded.

\medskip

In this paper we propose, as a first step into this direction, 
a Lie splitting scheme for abstract semilinear boundary coupled systems, where the semilinear term $\calF=(\calF_1,\calF_2)$ is locally Lipschitz (and might include feedback). 
An important feature of our splitting method is that it \emph{separates the flows on $E$ and $F$}, i.e.~separates the bulk and surface dynamics. This could prove to be a considerable computational advantage if the bulk and surface dynamics are fundamentally different (e.g.~fast and slow reactions, linear--nonlinear coupling, etc.).
In general, splitting methods simplify (or even make possible) the numerical treatment of complex systems. If the operator on the right-hand side of the initial value problem can be written as a sum of at least two suboperators, the numerical solution is obtained from a sequence of simpler subproblems corresponding to the suboperators. We will use the Lie splitting, introduced in \cite{BagGod57}, which, from the functional analytic viewpoint, corresponds to the Lie--Trotter product formula, see \cite{Trotter}, \cite[Corollary~III.5.8]{EngelNagel}. Splitting methods have been widely used in practice and analysed in the literature, see for instance the survey article \cite{McLachlanQuispel_acta}, and see also, e.g., \cite{Strang68,JahnkeLubich,Thalhammer2008,HanO09} and \cite{Faou15,EO_1,EO_2,EO_3,Alonso18,Alonso19,Alonso21}, etc. The latter group lists references which overcome order reduction of splitting schemes occurring for problems different form the boundary coupled system \eqref{eq:main0}. 

In particular, for semilinear partial differential equations (PDEs) with dynamic boundary conditions, two bulk--surface splitting methods were proposed in \cite{dynbc}. The numerical experiments of Section~6.3 therein illustrate that both of the proposed splitting schemes suffer from order reduction.

Recently, in \cite{dynbc_PDAE_splitting}, a first-order convergent bulk--surface Lie splitting scheme was proposed and analysed. In \cite{dynbc_PDAE_splitting} parabolic PDEs with dynamic boundary conditions were considered, proving that the splitting based fully discrete method converges to the spatially semi-discrete finite element solution. For results on semilinear problems in particular see \cite[Appendix~A]{dynbc_PDAE_splitting}. 
The following aspects separate the two approaches:
Even though both methods are suitable for semilinear problems, here we consider abstract boundary coupled parabolic problems in a general operator theoretic setting, see \eqref{eq:main0} and Section~\ref{sec:num}, while \cite{dynbc_PDAE_splitting} considers a particular family of PDEs on Lipschitz domains, see \cite[Section~2]{dynbc_PDAE_splitting}.
The two Lie splittings are derived via substantially different theoretical approaches: 
The one herein is derived using semigroup theory and exponential integrators in Banach spaces, see Section~\ref{sec:num}, in contrast the method in \cite{dynbc_PDAE_splitting} is derived based on the differential algebraic structure of the bulk--surface finite element semi-discretisation of a PDE (the resulting sub-flows are discretised using the implicit Euler method), see \cite[Section~4]{dynbc_PDAE_splitting}. 
Here, we assume local Lipschitz continuity of the nonlinearities, in \cite[Theorem~A.1]{dynbc_PDAE_splitting} global Lipschitz continuity and a mild CFL condition were assumed. 
Regularity requirements are comparable, for more details see Remark~\ref{remark:finalafterproof}.
We note that both methods differ from the two splitting schemes proposed in \cite{dynbc}. We are not aware of any other splitting-type numerical approaches separating the ``bulk'' and ``boundary'' dynamics.

In the present work we start by the variation of constants formula and apply the Lie splitting to approximate the appearing linear operator semigroups. More precisely, we will identify three linear suboperators: two describing the dynamics in the bulk and on the surface, respectively, and one corresponding to the coupling. Then, either the solutions to the linear subproblems are known explicitly, or can be efficiently obtained numerically. We will show that the proposed method is first-order convergent for boundary coupled semilinear problems, it does not suffer from order reduction, and is therefore promising for PDEs with dynamic boundary conditions, cf.~\cite{dynbc}, see the experiment in Section~\ref{section:numerics dynbc}.  Due to the \emph{unbounded} boundary feedback operator, our present analysis does not apply to this situation, however, our numerical experiments show first order convergence. We strongly believe that the developed techniques presented in this work provide further insight into the behaviour of operator splitting schemes of such problems. This is strengthened by our numerical experiments.

The convergence result is based on studying stability and consistency, using the procedure called Lady Windermere's fan from \cite[Section~II.3]{HairerWannerI}, however, these two issues cannot be separated as in most convergence proofs, since this would lead to sub-optimal error estimates. Instead, the error is rewritten using recursion formula which, using the parabolic smoothing property (see, e.g., \cite[Theorem~II.4.6 (c)]{EngelNagel}), leads to an induction process to ensure that the numerical solution stays within a strip around the exact solution. 
A particular difficulty lies in the fact that the numerical method for the linear subproblems needs to approximate a convolution term in the exact flow \cite{CsEF}, therefore the stability  of these approximations cannot be merely established based on semigroup properties. Estimates from \cite{CsEF} together with new technical results yield an abstract first-order error estimate for semilinear problems (with a logarithmic factor in the time step), under suitable (local Lipschitz-type) conditions on the nonlinearities. 
%
By this analysis within the abstract setting we gain a deep operator theoretical understanding of these methods, which are applicable for all specific models (e.g.~mentioned above) fitting into the framework of \eqref{eq:main0}.
Numerical experiments illustrate the proved error estimates, and an experiment for dynamic boundary conditions complement our theoretical results. 

\medskip
The paper is organised as follows.
In Section \ref{sec:num} we introduce the used functional analytic framework, and derive the proposed numerical method. We also state our main result, namely, the first-order convergence, the proof of which along with error estimates takes up Sections \ref{section:prep} and \ref{sec:proof}.
Finally, Section~\ref{section:numerics} presents numerical experiments illustrating and complementing our theoretical results.

\section{Setting and the numerical method}\label{sec:num}

We consider two Banach spaces $E$ and $F$, sometimes referred to as the bulk and boundary space, respectively, over the complex field $\CC$. The product space $E \times F$ is endowed with the sum norm, or any other equivalent norm, rendering it a Banach space and the coordinate projections bounded. Elements in the product space will be denoted by boldface letters, e.g.~$\vecu = (u,v)$ for $u\in E$ and $v\in F$. 
The norm on $E \times F$ is simply denoted by $\|\cdot\| = \|\cdot\|_E + \|\cdot\|_F$, on the other hand, for brevity, we will drop the subscript from the $E$ and $F$ norms, as it will be always clear from the context which norm is meant. We employ the usual convention for the operator norm as well.
Then we treat the nonlinearities, derive the numerical method, and present the main result of the paper.

\subsection*{General framework}

We will now define the abstract setting for \emph{linear} boundary coupled systems, established in \cite{CENN}, i.e.~for \eqref{eq:main0} with $\calF_1=0$ and $\calF_2=0$. We will also list all our assumptions on the linear operators in \eqref{eq:main0}.

The following general conditions---collected using Roman numerals---will be assumed throughout the paper: 
\begin{iiv}
	\item The (so-called maximal) operator $A_m\colon\dom(A_m)\subseteq E\to E$ is linear.
	\item The linear operator $L\colon\dom(A_m)\to F$ is surjective and bounded with respect to the graph norm of $A_m$ on $\dom(A_m)$.
	\item The restriction $A_0$ of $A_m$ to $\ker(L)$ generates a strongly continuous semigroup $T_0$ on $E$. The (linear) operator $A_0$ is invertible, i.e., $0\in\rho(A_0)$ the resolvent set of $A_0$, (cf.~\cite[Lemma~2.2]{CENN}).
	\item The  (linear) operator $B$ generates a strongly continuous semigroup $S$ on $F$.
	\item The operator matrix $\vect{A_m}{L}\colon\dom(A_m)\to E\times F$ is closed.
\end{iiv}

We recall from \cite[Lemma 2.2]{CENN} that $L|_{\ker(A_m)}$ is invertible, and its inverse, often called the \emph{Dirichlet operator}, given by
\begin{align}\label{eq:D}
D_0 &:= L|_{\ker(A_m)}^{-1}\colon F\to\ker(A_m)\subseteq E ,
\end{align}
is bounded, and (via $\ker(L) \subset \dom(A_m)$) that 
\[
\dom(A_m)=\dom(A_0)\oplus \ker(A_m) = \ker(L) \oplus \ker(A_m) .
\]
Note that the range of the Dirichlet operator $D_0$ is almost disjoint from the domain of $A_0$ (they intersect in $\{0\}$). Therefore no matter how regular the initial conditions are, the action of $D_0$ serves as a barrier, and will result in the logarithmic loss in the error estimate for exact solutions of arbitrary smoothness.

Let us briefly recall the following example from \cite{CsEF} (see  Examples 2.7 and 2.8 therein), which is also one of the main motivating examples of \cite{CENN}; we refer also to \cite{GeMit11, GeMit08, BeGeMi20} for facts concerning Lipschitz domains.

\begin{example}[Bounded Lipschitz domains]\label{examp:Lip}
		Let $\Omega\subseteq \RR^d$ be a bounded domain with Lipschitz boundary $\Gamma$, $E:=\Ell^2(\Omega)$ and $F:=\Ell^2(\Gamma)$.
	\begin{abc} 
		\item Consider the following operators: $A_m=\Delta_\Omega$ with domain 
		$\dom(A_m):=\{f\st f\in \He^{1/2}(\Omega)\text{ with }\Delta_\Omega f\in \Ell^2(\Omega)\}$ (being the maximal, distributional Laplacian $A_m=\Delta_\Omega$ without boundary conditions), and $Lf=f|_{\Gamma}$ the Dirichlet trace of $f\in \dom(A_m)$ on $\Gamma$ (see, e.g., \cite[pp.{} 89--106]{McL00}).
		Then $L$ is surjective and actually has a bounded right-inverse $D_0$, which is   the harmonic extension operator, i.e.~for any $v \in \Ell^2(\Gamma)$ the function $u = D_0 v$ solves (uniquely) the  Poisson problem $\Delta_\Omega u = 0$ with inhomogeneous Dirichlet boundary condition $L u = v$. The operator $A_0$ is strictly negative, self-adjoint and generates the Dirichlet-heat semigroup $T_0$ on $E$.
		
		\item One can also consider the Laplace--Beltrami operator $B:=\Delta_{\Gamma}$ on $\Ell^2(\Gamma)$, which (with an appropriate domain) is also a strictly negative, self-adjoint operator, see \cite[Theorem 2.5]{GeMit11} or \cite{GeMiMiMi} for details.
	\end{abc}
	In summary, we see that the abstract framework of \cite{CENN}, hence of this paper, covers interesting cases of boundary coupled problems on bounded Lipschitz domains.
\end{example} 

We now turn our attention towards the semigroup, and its generator, corresponding to the linear problem. 
Consider the linear operator 
\begin{equation}\label{eq:calA}
\calA:=\begin{pmatrix}A_m&0\\0& B\end{pmatrix}
\quad \text{with} \quad 
\dom(\calA):=\Big\{\vect{x}{y}\in\dom(A_m)\times\dom(B)\st Lx=y\Big\}.
\end{equation}
For $y\in \dom(B)$ and $t\geq 0$ define the convolution
\begin{equation}\label{eq:Q0tdef}
Q_0(t)y:=-\int_0^t T_0(t-s)D_0S(s)By\dd s.
\end{equation}
For all $y\in\dom(B)$ we also define $Q(t)y$,  and using integration by parts, see \cite{CENN}, we immediately write
\begin{equation}\label{eq:Qtdef}
Q(t)y:=-A_0\int_0^t T_0(t-s)D_0S(s)y\dd s=Q_0(t)y+D_0S(t)y-T_0(t)D_0y.
\end{equation}
We see that $Q_0(t)\colon\dom(B)\to E$ and $Q(t)\colon\dom(B)\to E$ are both linear operators on $\dom(B)$ (and we note that they are bounded when $\dom(B)$ is endowed with the graph norm $\|\cdot\|_B := \|\cdot\| + \|B \cdot \|$).

The next result, recalled from \cite{CENN}, characterizes the generator property of $\calA$, which in turn is in relation with the well-posedness of \eqref{eq:main0}, see Section~1.1 in \cite{MugnoloDiss}.
\begin{theorem}[{\cite[Theorem 2.7]{CENN}}]\label{thm:Engel}
	Within this setting, let the operators $\calA$, $D_0$ be as defined in \eqref{eq:calA} and \eqref{eq:D}, and suppose that $A_0$ is invertible. 
	The operator $\calA$ is the generator of a $C_0$-semigroup 
	if and only if for each $t\ge 0$ the operator $Q(t) \colon \dom(B) \to E$ extends to a bounded linear operator $Q(t)\colon F \to E$ with its operator norm satisfying
	\begin{equation}
	\label{eq:Qtb}
	\limsup_{t\downto 0}\|Q(t)\|<\infty.
	\end{equation}
	The semigroup $\calT$ generated by $\calA$ is then given as
	\begin{equation}
	\label{eq:calT}
	\calT(t)=\begin{pmatrix}T_0(t)& Q(t)\\0 & S(t)\end{pmatrix}.
	\end{equation}
\end{theorem} 
In other words, if the conclusion of Theorem~\ref{thm:Engel} holds, then  the linear problem $\dot \vecu = \calA \vecu$ is well-posed and the solution with initial value $\vecu_0 = (u_0,v_0)$ is given by the semigroup as $\calT(t)\vecu_0$.

We augment the list of general conditions (i)--(v) by further assuming: 
\begin{iiv}\setcounter{enumi}{5}
	\item Additionally to $A_0$ also operator $B$ is invertible. 
	\item The operators $A_0$ and $B$ generate boun\-ded analytic semigroups. 
\end{iiv}
The operators appearing in Example~\ref{examp:Lip} satisfy all of these conditions.

\begin{remark}
	\label{remark:analytic remarks}
	\begin{abc}
		\item By Corollary 2.8 in \cite{CENN} the assumption in (vii) implies that $\calA$ is the generator of an analytic $C_0$-semigroup on $E\times F$.
		\item The invertibility of $A_0$ or $B$ is merely a technical assumption which slightly simplifies the proofs and assumptions, avoiding a shifting argument.
		\item In principle, one can drop the assumption of $B$ being the generator of an analytic semigroup. In this case minor additional assumptions on the nonlinearity $\calF$ are needed, and the error bound for the numerical method will look slightly differently. We will comment on this in Remark~\ref{remark:finalafterproof} below, after the proof of the main theorem.
		\item The fact that $A_0$ generates a bounded analytic semigroup $T_0$ implies the bound
		$\sup_{t\geq0}\|tA_0T_0(t)\|\leq M$, see, e.g., \cite[Theorem~II.4.6 (c)]{EngelNagel}.
	\end{abc}
\end{remark}

For further details on analytic semigroups we refer to the monographs \cite{Pazy,Lunardi,EngelNagel,Haase}.

\subsection*{The abstract semilinear problem}

We now turn our attention to semilinear boundary coupled problems \eqref{eq:main0}.
In particular we will give our precise assumptions related to the solutions of the semilinear problem, and to the nonlinearity $\calF=(\calF_1,\calF_2)\colon \dom(\calF)\subseteq E\times F \to E \times F$.

\begin{assumptions}
	\label{ass:solution} 
	The function $\vecu:=(u,v)\colon[0,\tmax]\to E\times F$, $\tmax>0$, is a mild solution of the problem \eqref{eq:main0}, written on $E \times F$ as
	\begin{equation}
	\label{eq:abstract evolution equation}
	\dot \vecu = \calA \vecu + \calF(\vecu) ,
	\end{equation}
	i.e.~it satisfies the variation of constant formula:
	\begin{equation}\label{eq:voc}
	\vecu(t) = \calT(t)\vecu_0+\int_0^t\calT(t-s)\calF(\vecu(s))\dd s.
	\end{equation}
	We further assume that the exact solution $\vecu$ has the following properties:
	\begin{numerA}
		\item \label{item:1} The function $\calF:\str\to E\times F$ is Lipschitz continuous on the strip, for an $R>0$, 
		\begin{equation*}
		\str:= \str_R := \{\vecv\in E\times F\st \|\vecu(t)-\vecv\|\leq R\text{ for some $t\in [0,\tmax]$}\}\subseteq \dom(\calF)
		\end{equation*}
		around the exact solution with constant $\LIP$.
		\item \label{item:2} The second component $\calF_2:\str\to \dom(B)$ is Lipschitz continuous on $\str$, with constant $\LIPB$. 
		\item \label{item:3} For each $t\in [0,\tmax]$ $v(t)=\vecu(t)|_2\in \dom(B^2)$, and $\sup_{t\in[0,\tmax]}\|B^2v(t)\|<\infty$ (here $|_2$ denotes the projection onto the second component). 
		\item \label{item:4} The second component along the solution satisfies $\calF_2(\vecu(t))\in \dom(B^2)$ for each $t\in [0,\tmax]$, and $\sup_{t\in[0,\tmax]}\|B^2\calF_2(\vecu(t))\|<\infty.$
		\item \label{item:5} Furthermore, $\calF\circ \vecu$ is differentiable and $(\calF\circ \vecu)'\in \Ell^1([0,\tmax];E \times F)$. %
	\end{numerA}
\end{assumptions}

The local Lipschitz continuity in strip is ensured if $\calF$ is assumed to be locally Lipschitz continuous on bounded sets, see, e.g., \cite{HochbruckOstermann} and \cite{dynbc}. Later in the convergence analysis we will show that the numerical solution stays sufficiently close to the exact solution. We further comment on these assumptions, and their relation to those in \cite{dynbc_PDAE_splitting}, after the proof in Remark~\ref{remark:finalafterproof}. Later on in the convergence proof, we will directly reference these properties by their number ($\cdot$).

\subsection*{The numerical method}

We are now in the position to derive the numerical method. For a time step $\tau>0$, for all $t_n = n\tau\in[0,\tmax]$, we define the numerical approximation $\vecu_n=(u_n,v_n)$ to $\vecu(t_n)=(u(t_n),v(t_n))$ via the following steps.

\begin{enumsteps}
	\item We approximate the integral in \eqref{eq:voc} by an appropriate quadrature rule.
	\item We approximate the semigroup operators $\calT$ by using an operator splitting method. Due to its special form \eqref{eq:calT}, this includes the approximation of the convolution $Q_0$, defined in \eqref{eq:Qtdef}, by an operator $V$. The choice of $V$ is determined by the used splitting method, see \cite[Section~3]{CsEF} and below.
\end{enumsteps}
In what follows we describe the numerical method by using first-order approximations in \emph{Steps 1--2}, and show its first-order convergence. We note here that the application of a correctly chosen exponential integrator could be inserted as a preliminary step, see \cite{HochbruckOstermann}. Since it eliminates the integral's dependence on $\vecu(s)$, the quadrature rule simplifies in \emph{Step 1}. This approach, however, leads to the same numerical method as \emph{Steps 1--2}.

\medskip
Before proceeding as proposed, for all $\tau>0$, we rewrite formula \eqref{eq:voc} at $t=t_n=t_{n-1}+\tau$ as
\begin{align}
\label{eq:VOC}
\vecu(t_n) = \calT(\tau)\vecu(t_{n-1})+\int_0^\tau\calT(\tau-s)\calF(\vecu(t_{n-1}+s))\dd s.
\end{align}
Now, according to \emph{Step 1}, we approximate the integral by the left rectangle rule leading to
\begin{equation*}
\vecu(t_n)\approx \calT(\tau)\vecu(t_{n-1})+ \tau\calT(\tau)\calF(\vecu(t_{n-1}))= \calT(\tau) \Big( \vecu(t_{n-1})+\tau\calF(\vecu(t_{n-1})) \Big) ,
\end{equation*}
for any $t_n = n\tau\in(0,\tmax]$. 

In \emph{Step 2}, we apply the Lie splitting, which, according to \cite{CsEF}, results in the approximation of the convolution operator $Q_0(t)$ by an appropriate $V(t)$ (to be specified later). Altogether, we approximate the semigroup operators $\calT(\tau)$ by
\begin{equation}\label{eq:def calTT}
\calTT(\tau) = \begin{pmatrix} T_0(\tau)&V(\tau) + D_0 S(\tau)-T_0(\tau)D_0\\0&S(\tau)\end{pmatrix}.
\end{equation}
We remark that $\calTT(\tau)=\calR_0^{-1}\TT(\tau)\calR_0$ holds with the notations introduced in \cite{CsEF}:
\begin{equation*}
\TT(\tau)=\begin{pmatrix} T_0(\tau)&V(\tau)\\0&S(\tau) \end{pmatrix} \quad\text{and}\quad
\calR_0=\begin{pmatrix} I&-D_0\\0&I \end{pmatrix}.
\end{equation*}
This leads to the numerical method approximating $\vecu$ at time $t_n = n\tau \in [0,\tmax]$:
\begin{equation}
\label{eq:method1}
\vecu_n 
:= \bfL(\tau)(\vecu_{n-1}) 
:= \calTT(\tau)\big(\vecu_{n-1}+\tau\calF(\vecu_{n-1})\big) ,
\end{equation}
with $\vecu_0:=(u_0,v_0)$.

The actual form of operator $V(\tau)$ depends on the underlying splitting method. Here, we will use the Lie splitting of the operator $\calA_0:=\calR_0\calA\calR_0^{-1}$, proposed in \cite[Section~3]{CsEF}. Notice that
\begin{equation*}
\calA_0 = \begin{pmatrix}A_0&-D_0B\\0&B\end{pmatrix}\quad\text{with\quad $\dom(\calA_0)=\dom(A_0)\times \dom(B)$},
\end{equation*}
i.e., the similarity transformation $\calR_0$, diagonalizes the domain of the formally diagonal operator $\calA$ with non-diagonal domain by decoupling the two components, cf.{} \eqref{eq:calA}. \blueon Indeed, $\vect{x}{y}\in\dom(\calA_0)$ if and only if $\vect{x+D_0y}{y}=\calR_0^{-1}\vect{x}{y}\in\dom(\calA)$, i.e., if $x+D_0y\in \dom(A_m)$, $L(x+D_0y)=y$  and $y\in \dom(B)$. These hold if and only if $x\in \dom(A_m)$ (since $D_0y\in \dom(A_m)$), $x\in \ker(L)$ (since $LD_0y=y$) and $y\in \dom(B)$; this proves that $\calA_0$ has diagonal domain, and also the above form of $\calA_0$ follows at once, since for $x\in \ker(L)$ one has $A_0x=A_mx$. Now the  prize to be paid for this diagonalization of the domain \blueoff is the new unbounded term $-D_0B$ making $\calA_0$ upper triangular. Notice also that the decoupling via the similarity transformation changes the diagonal term $A_m$ into $A_0$. This technique is an abstract operator theoretic version of homogenizing boundary conditions.
Now, the splitting is  given by $\calA_0=:\calA_1+\calA_2+\calA_3$ with 
\begin{equation*}
\calA_1 = \begin{pmatrix}A_0&0\\0&0\end{pmatrix},
\quad 
\calA_2 = \begin{pmatrix}0&-D_0B\\0&0\end{pmatrix}, 
\quad 
\calA_3 = \begin{pmatrix}0&0\\0&B\end{pmatrix},
\end{equation*}
and $\dom(\calA_1)=\dom(A_0)\times F$, $\dom(\calA_2)=E\times\dom(B)$, $\dom(\calA_3)=E\times\dom(B)$. It was shown in \cite[Prop.~3.2.]{CsEF} that the operator \emph{parts}, see \cite[Section II.2.3]{EngelNagel}, $\calA_1|_{E\times \dom(B)}$, $\calA_2|_{E\times \dom(B)}=\calA_2$ and $\calA_3|_{E\times \dom(B)}$ generate the strongly continuous semigroups\footnote{Note that the domain of definition of the part $\calA_3|_{E\times \dom(B)}$ is $E\times \dom(B^2)$, and that, in fact, $\calA_2|_{E\times \dom(B)}=\calA_2$.} 
\begin{equation*}
\calT_1(\tau) = \begin{pmatrix}T_0(\tau)&0\\0&I\end{pmatrix}, \quad \calT_2(\tau) = \begin{pmatrix}I&-\tau D_0B\\0&I\end{pmatrix}, \quad \calT_3(\tau) = \begin{pmatrix}I&0\\0&S(\tau)\end{pmatrix} , 
\end{equation*}
respectively, on $E\times\dom(B)$. Then the application of the Lie splitting as $\calTT(\tau)=\calR_0^{-1}\calT_1(\tau)\calT_2(\tau)\calT_3(\tau)\calR_0$ leads to the formula \eqref{eq:def calTT} with
\begin{equation}\label{eq:def V}
V (\tau)= -\tau T_0(\tau)D_0BS(\tau).
\end{equation}
Thus, the Lie splitting transfers the coupled linear problem into the sequence of simpler ones. First we solve the equation $\dot v=Bv$ on $\dom(B)$ by using the original initial condition $v_0$, then we propagate the solution by $\calT_2(\tau)$, which serves as an initial condition to the homogeneous problem $\dot u=A_0u$ on $E$. To get an approximation at $t_n=n\tau$, the semilinear expressions and the terms coming from the ``diagonalisation'' should be treated. Then the whole process needs to be cyclically performed $n$ times.

\medskip
We note that the approximation $Q_0(\tau)\approx V(\tau)=-\tau T_0(\tau)D_0BS(\tau)$ can also be obtained by using an appropriate convolution quadrature, i.e.~by approximating $T_0(\tau-\xi)$ from the left (at $\xi=0$) and $S(\xi)$ from the right (at $\xi=\tau$).

\medskip
Upon plugging in the splitting approximation \eqref{eq:def V} into the convolution $Q_0(\tau)$, and by introducing the intermediate values
\begin{subequations}
	\label{eq:methodcomp min D_0}
	\begin{equation}
	\begin{aligned}
	\widetilde u_n &= u_{n-1}+\tau \calF_1(u_{n-1},v_{n-1}), \\
	\widetilde v_n &= v_{n-1}+\tau \calF_2(u_{n-1},v_{n-1}),
	\end{aligned}
	\end{equation} 
	the method \eqref{eq:method1} reads componentwise as
	\begin{equation}
	\begin{aligned}
	u_n = &\ T_0(\tau) \Big( \widetilde u_{n} - D_0 \big( \widetilde v_{n} + \tau B v_{n} \big) \Big) + D_0 v_n , \\
	v_n = &\ S(\tau) \widetilde v_n.
	\end{aligned}
	\end{equation}
\end{subequations}
This formulation only requires two applications of the Dirichlet operator $D_0$ per time step. 
We point out that the two terms with the Dirichlet operator can be viewed as correction terms which correct the boundary values of the bulk-subflow along the splitting method.

\subsection*{The main result}

We are now in the position to state the main result of this paper, which asserts first order (up to a logarithmic factor) error estimates for the approximations obtained by the splitting integrator \eqref{eq:method1} (with \eqref{eq:def V}), i.e.~\eqref{eq:methodcomp min D_0}, separating the bulk and surface dynamics in $E$ and $F$.

\begin{theorem}
	\label{theorem:Lie}
	In the above setting, 
	let $\vecu:[0,\tmax]\to E\times F$ be the solution of \eqref{eq:main0} subject to the conditions in Assumptions~\ref{ass:solution} and consider the approximations $\vecu_n$ at time $t_n$ determined by the splitting method \eqref{eq:method1} (with \eqref{eq:def V}), or written in componentwise form \eqref{eq:methodcomp min D_0}. Then there exists a $\tau_0>0$ and $C>0$ such that for any time step $\tau \leq \tau_0$ we have at time $t_n=n\tau \in [0,\tmax]$ the error estimate
	\begin{equation}
	\label{eq:Lie error estimate}
	\|\vecu(t_n)-\vecu_{n}\| \leq C \, \tau \, |\log(\tau)| .
	\end{equation}
	The constant $C>0$ is independent of $n$ and $\tau>0$, but depends on $\tmax$, on constants related to the semigroups $T_0$ and $S$, as well as on the exact solution $\vecu$. 
\end{theorem}

The proof of this result will be given in Section \ref{sec:proof} below. In the next section we state and prove some preparatory and technical results needed for the error estimates.

Recall that the splitting method \eqref{eq:method1}, written componentwise \eqref{eq:methodcomp min D_0}, decouples the bulk and surface flows, which can be extremely advantageous if the two subsystems behave in a substantially different manner. We remind that, when applied to PDEs with dynamic boundary conditions, naive splitting schemes suffer from order reduction, see \cite[Section~6]{dynbc}, and a correction in \cite{dynbc_PDAE_splitting}.

We make the following remark about the logarithmic factor  in the above error estimate. Inequality \eqref{eq:Lie error estimate} implies that for any $\varepsilon\in (0,1)$	we have $\|\vecu(t_n)-\vecu_{n}\| \leq C' \, \tau^{1-\varepsilon}$ with another constant $C'$. This amounts to saying that the proposed method has convergence order arbitrarily close to $1$, and in fact this is also what the numerical experiments show. 
Indeed, numerical experiments in Section~\ref{section:numerics} illustrate and complement the first-order error estimates of Theorem~\ref{theorem:Lie}, including an example with \emph{dynamic boundary conditions}, without any order reductions.

\section{Preparatory results}
\label{section:prep}

In this section we collect some general technical results which will be used later on in the convergence proof.
After a short calculation, or by using the results in  Section~3 of \cite{CsEF}, we obtain
\begin{align}
\label{eq:powers of calTT}
\calTT(\tau)^k 
= &\
\begin{pmatrix}
T_0(k\tau) & -T_0(k\tau) D_0 + D_0 S(k\tau) + V_k(\tau) \\
0 & S(k\tau) 
\end{pmatrix} , \\ 
\nonumber
\text{where} \qquad V_k(\tau) y 
= &\ \sum_{j=0}^{k-1} T_0((k-1-j)\tau) V(\tau) S(j\tau) y,
\end{align}
see~\cite[equation~(3.9)]{CsEF}.
Now we are in the position to prove exponential bounds for the powers of $\calTT(\tau)$.
\begin{lemma}\label{lem:calTTpower}
	There exist a constant $M > 0$  such that for $\tau > 0$ and $\calTT(\tau)$ defined in \eqref{eq:def calTT} (with \eqref{eq:def V}), and  for any $(x,y) \in E \times \dom(B)$ and $k \in \NN$ with $k\tau\in [0,\tmax]$
	\begin{equation*}
	\|\calTT(\tau)^k \vect{x}{y} \| \leq M \|\vect{x}{y}\| + M   \|B y\| .
	\end{equation*}
	Moreover, if $S$ is a bounded analytic semigroup, then we have
	\begin{equation*}
	\|\calTT(\tau)^k \vect{x}{y} \| \leq M (1+\log(k)) \|\vect{x}{y}\|.
	\end{equation*}
\end{lemma}
\begin{proof}
	For the sum norm on the product space $E\times F$, we have 
	\begin{align*}
	\|\calTT(\tau)^k \vect{x}{y} \| &= \|T_0(k\tau) x-T_0(k\tau) D_0 y+D_0 S(k\tau) y+V_k(\tau)y\| + \|S(k\tau)y\| \\
	& \le \|T_0(k\tau) x\| + \|T_0(k\tau) D_0 y\|  + \|D_0 S(k\tau) y\| + \|V_k(\tau)y\| + \|S(k\tau)y\|.
	\end{align*}
	The exponential boundedness of the semigroups $T_0$ and $S$, and the boundedness of $D_0$ directly yield
	\begin{align*}
	\|T_0(k\tau) x\| + \|T_0(k\tau) D_0 y\|  + \|D_0 S(k\tau) y\| &\leq M \big( \|x\| + \|y\| \big) , \\  
	\text{and} \quad \|S(k\tau) y\| &\leq M  \|y\|.
	\end{align*}
	It remains to bound the term $V_k(\tau) y$. We obtain
	\begin{align*}
	\|V_k(\tau) y\|
	\leq &\ \tau \sum_{j=0}^{k-1} \| T_0((k-1-j)\tau) T_0(\tau) D_0 B S(\tau) S(j\tau) y \| \\
	\leq &\ \tau \sum_{j=0}^{k-1} \| T_0((k-j)\tau)  D_0 S((j+1)\tau) B y \| \\
	\leq &\ \tau \sum_{j=0}^{k-1} M_1   \|B y\| \leq \tmax  M_1   \|B y\|
	\leq  M \|B y\| ,
	\end{align*}
	which completes the proof of the first statement.
	
	\medskip If $S$ is a bounded analytic semigroup, then we improve the last estimate to
	\begin{align*}
	\|V_k(\tau)\| &= \sum\limits_{j=0}^{k-1}\|T_0\bigl((k-1-j)\tau\bigr)V(\tau)S(j\tau)\|\\
	&=	\tau\sum\limits_{j=0}^{k-1}\|T_0\bigl((k-j)\tau\bigr)\| \, \|D_0BS(\tau)S(j\tau)\|\\
	&\leq M_1M_2\|D_0\| \, \tau \sum_{j=0}^{k-1}\frac{1}{(j+1)\tau}\leq M(1+\log(k)).
	\end{align*}
	By putting the estimates together, the assertions follows.
\hfill \end{proof}

We recall the following lemma from \cite{CsEF}.
\begin{lemma}[{\cite[Lemma 4.4]{CsEF}}]\label{lem:loclieA}
	There is a $C\ge 0$ such that for every $\tau\in[0,\tmax]$, 
	for any $s_0,s_1\in [0,\tau]$, and for every $y\in \dom(B^2)$ we have
	\begin{equation*}
	\Bigl\|\int_{0}^{\tau} T_0(\tau-s)A_0^{-1}D_0 S(s)By\dd s-\tau T_0(\tau-s_0)A_0^{-1}D_0 S(s_1)By\Bigr\|
	\le C \tau^{2}(\|By\|+\|B^2 y\|).
	\end{equation*} 
\end{lemma}

\noindent Using the above quadrature estimate we prove the following approximation lemma.
\begin{lemma}\label{lem:delta1}
	For $(x,y)\in E\times\dom(B^2)$ and $j\in\NN\setminus\{0\}$ we have
	\begin{equation*}
	\big\|\calTT(\tau)^j \big( \calT(\tau) - \calTT(\tau) \big) \vect{x}{y} \big\|\le C \tau^2 \|A_0 T_0(j\tau)\| \, \big(\|By\|+\|B^2y\|\big).
	\end{equation*}
\end{lemma}
\begin{proof}
	Using the formula \eqref{eq:powers of calTT} for $\calTT(\tau)^j$ and a direct computation for the difference $\calT(\tau) - \calTT(\tau)$, we obtain
	\begin{align*} 
	&\ \calTT(\tau)^j \big( \calT(\tau) - \calTT(\tau) \big) \vect{x}{y} \\
	= &\ \calTT(\tau)^j \ \bigg(- \int_0^\tau T_0(\tau-\xi) D_0 B S(\xi)y\dd\xi + \tau T_0(\tau) D_0 B S(\tau)y  \ , \ 0 \ \bigg)^\top \\
	= &\ \bigg( - T_0(\tau)^j \bigg( \int_0^\tau T_0(\tau-\xi) D_0 B S(\xi)y\dd\xi -\tau T_0(\tau) D_0 B S(\tau)y \bigg)\ , \ 0 \ \bigg)^\top
	\end{align*}
	for all $(x,y)\in E\times\dom(B)$. We can further rewrite the first component as
	\begin{align*} 
	&\ T_0(j\tau) \bigg( \int_0^\tau T_0(\tau-\xi) D_0 B S(\xi)y\dd\xi -\tau T_0(\tau) D_0 B S(\tau)y \bigg) \\
	= &\ A_0 T_0(j\tau) \, \bigg( \int_0^\tau T_0(\tau-\xi) A_0^{-1}D_0 B S(\xi)y\dd\xi -\tau T_0(\tau) A_0^{-1}D_0 B S(\tau)y \bigg).
	\end{align*}
	We have
	\begin{align*}
	&\ \big\| \calTT(\tau)^j \big( \calT(\tau) - \calTT(\tau) \big) \vect{x}{y} \big\|  \\
	= &\ \bigg\| A_0 T_0(j\tau) \, \bigg( \int_0^\tau T_0(\tau-\xi) A_0^{-1}D_0 B S(\xi)y\dd\xi -\tau T_0(\tau) A_0^{-1}D_0 B S(\tau)y\bigg) \bigg\| \\
	\le &\ \|A_0 T_0(j\tau)\| \, \bigg\|\int_0^\tau T_0(\tau-\xi) A_0^{-1}D_0 B S(\xi)y\dd\xi -\tau T_0(\tau) A_0^{-1}D_0 B S(\tau)y\bigg\|,
	\end{align*}
	therefore an application of Lemma~\ref{lem:loclieA} with $s_0=0$ and $s_1=\tau$ proves the assertion.
\hfill \end{proof}

\begin{lemma}\label{lem:firsttriv}
	For $t,s\in [0,\tmax]$ we have
	\begin{align*}
	\|A_0^{-1}T_0(t)-A_0^{-1}T_0(s)\|\leq M|t-s|.
	\end{align*}
\end{lemma}
\begin{proof}
	Resorting to the Taylor expansion we have for $x\in E$ that
	\begin{equation*}
	A_0^{-1}T_0(t)x-A_0^{-1}T_0(s)x=\int_s^t T_0(r)A_0^{-1}A_0x\dd r=\int_s^t T_0(r)x\dd r ,
	\end{equation*}
	which readily implies $\|A_0^{-1}T_0(t)x-A_0^{-1}T_0(s)x\|\leq M\|x\||t-s|$, and hence the assertion.
\hfill \end{proof}

\begin{lemma}\label{lem:thirdtriv}
	Let $f\colon[0,\tmax]\to E$ be Lipschitz continuous and consider
	\begin{equation*}
	(T_0*f)(t):=\int_0^t T_0(t-r)f(r)\dd r,\quad t\in [0,\tmax].
	\end{equation*}
	Then for all $t,s\in[0,\tmax]$ we have
	\begin{equation*}
	\|(T_0*f)(t)-(T_0*f)(s)\|\leq C|t-s|\|f\|_\Lip.
	\end{equation*}
\end{lemma}
\noindent Here, by $\|\cdot\|_\Lip$ we denote the norm $\|f\|_\Lip := \|f\|_\infty + \ell_{\text{best}}(f)$, where $
\ell_{\text{best}}(f)\geq  0$ is the best Lipschitz constant of $f$. 
\begin{proof}
	For $t,s\in[0,\tmax]$, we have
	\begin{align*}
	& \big\| (T_0*f)(t)-(T_0*f)(s)\big\|  
	= \Bigl\|\int_0^t T_0(r)f(t-r)\dd r-\int_0^sT_0(r)f(s-r)\dd r\Bigr\|\\
	& \leq \int_0^s\|T_0(r)(f(t-r)-f(s-r))\|\dd r+\int_s^t\|T_0(r)f(t-r)\|\dd r\\
	& \leq \Co_1 |t-s|s\|f\|_\Lip+\Co_1|t-s|\|f\|_\infty\leq C|t-s|\|f\|_\Lip.
	\end{align*}
\hfill \end{proof}


Let $\vert_1$ and $\vert_2$ denote the projection onto the first and second coordinate in $E\times F$.

\begin{lemma}\label{lem:delta3}
	For $\tmax>0$ there is a $C\geq 0$ such that for every $(x,y)\in E\times\dom(B)$, $t,s\in [0,\tmax]$ we have
	\begin{align*}
	\bigl\|\big(\calT(t)-\calT(s)\big)\vect{x}{y}\vert_1\bigr\|& \leq C \, (\|x\|+\|y\|+\|By\|), \\
	\text{and} \qquad \bigl\|\big(\calT(t)-\calT(s)\big)\vect{x}{y}\vert_2\bigr\| &\leq C\,|t-s|\|By\|.
	\end{align*}
\end{lemma}
\begin{proof}
	We have
	\begin{align*}
	\big(\calT(t)-\calT(s)\big)\vect{x}{y}\vert_2=\int_s^{t}S(r)By\dd r
	\end{align*}
	and the second asserted inequality follows at once. 
	
	On the other hand, for the first component
	\begin{align*}
	&\big(\calT(t)-\calT(s)\big)\vect{x}{y}\vert_1=T_0(t)x-T_0(s)x+Q(t)y-Q(s)y\\
	&=T_0(t)x-T_0(s)x+D_0S(t)y-D_0S(s)y-T_0(t)D_0y+T_0(s)D_0y+Q_0(t)y-Q_0(s)y,
	\end{align*}
	and we obtain
	\begin{align*}
	\big\|\big(\calT(t)-\calT(s)\big)\vect{x}{y}\vert_1\big\|\le 2M\|x\| + 4M\|D_0\|\|y\|+(|t|+|s|) M^2\|D_0\|\|By\|,
	\end{align*}
	and the first inequality is also proved. 
\hfill \end{proof}

\begin{lemma}\label{lem:delta2}
	For $\tmax>0$ there is a $C\geq 0$ such that for every $(x,y)\in E\times\dom(B^2)$, $t,s\in [0,\tmax]$, $\tau>0$, $0 < j\tau\le\tmax$ we have
	\begin{align*}
	\bigl\|\calTT(\tau)^j\big(\calT(t)-\calT(s)\big)\vect{x}{y}\bigr\| \leq&  C\,|t-s|\|A_0T_0(j\tau)\|  (\|x\|+\|y\|+\|By\|)\\
	&+C \, |t-s| (\|By\|+\|B^2y\|).
	\end{align*}
\end{lemma}
\begin{proof}
	From \eqref{eq:powers of calTT} we obtain
	\begin{align*}
	\calTT(\tau)^j\big(\calT(t)-\calT(s)\big)\vect{x}{y}\vert_2= &\ \int_s^{t}S(j\tau+r)By\dd r , \quad\qquad \text{and} \\
	\calTT(\tau)^j\big(\calT(t)-\calT(s)\big)\vect{x}{y}\vert_1
	= &\ T_0(j\tau )\bigl(T_0(t)x-T_0(s)x+Q(t)y-Q(s)y\bigr) \\
	 &\ - T_0(j\tau ) D_0\int_s^{t}S(r)By\dd r+D_0 S(j\tau ) \int_s^{t}S(r)By\dd r + V_j(\tau) \int_s^{t}S(r)By\dd r\\
	 = &\ I_1+I_2+I_3+I_4,
	\end{align*}
	where $I_1,\dots, I_4$ denote the four terms in the order of appearance. By Lemma \ref{lem:firsttriv}  
	\begin{align*}
	\|I_1\|&\leq \|A_0 T_0(j\tau )\| \Bigl(\|A_0^{-1}(T_0(t)-T_0(s))\|\|x\|+ \|A_0^{-1}(Q(t)-Q(s))y\|\Bigr)\\
	&\leq C \|A_0 T_0(j\tau )\| |t-s|\|x\|+\|A_0 T_0(j\tau )\|\|A_0^{-1}(Q(t)-Q(s))y\|,
	\end{align*}
	so we need to estimate $\|A_0^{-1}(Q(t)-Q(s))y\|$. 
	Since $A_0^{-1}Q$ has the appropriate convolution form, Lemma \ref{lem:thirdtriv} implies
	\begin{align*}
	&\|A_0^{-1}(Q(t)-Q(s))y\| =\Bigl\| (T_0*D_0S)(t)-(T_0*D_0S)(s)\Bigr\|\leq \Co_1|t-s|\|D_0\|\|By\|.
	\end{align*}
	Altogether we obtain 
	\begin{equation*}
	\|I_1\|\leq \Co_2|t-s| \|A_0 T_0(j\tau )\|(\|x\|+\|y\|+\|By\|).
	\end{equation*}
	For $I_2$ and $I_3$ we have
	\begin{equation*}
	\|I_2\|+\|I_3\|\leq \Co_3|t-s|\|By\|.
	\end{equation*}
	To estimate $I_4$ we recall from the proof of Lemma \ref{lem:calTTpower} that $\|V_j(\tau)z\|\leq \Co_4\|Bz\|$ (for $j\tau\leq [0,\tmax]$), so 
	\begin{equation*}
	\|I_4\|\leq \Co_4 \Bigl\|B \int_s^{t}S(r)By\dd r\Bigr \|\leq \Co_5|t-s| \|B^2y\|.
	\end{equation*}
	Finally, the estimates for $I_1,\dots,I_4$ together yield the assertion.
\hfill \end{proof}

\section{Proof of Theorem~\ref{theorem:Lie}}
\label{sec:proof}

The proof of our main result is based on a recursive expression for the global error, which involves the local error and some nonlinear error terms.
The recursive formula is obtained using a procedure which is sometimes called Lady Windermere's fan \cite[Section~II.3]{HairerWannerI}; our approach is inspired by \cite{OPiazzolaWalach}, \cite[Chapter~3]{diss_Walach}.
The local errors are weighted by $\calTT(\tau)^j$, therefore a careful accumulation estimate---heavily relying on the parabolic smoothing property---is required. In order to estimate the locally Lipschitz nonlinear terms we have to ensure that the numerical solution remains in the strip $\str$ (see Assumptions~\ref{ass:solution}). This will be shown using an induction process, which is outlined as follows:
\begin{itemize}
	\item We shall find $\tau_0 > 0$ and a constant $C>0$ such that for any $0 < \tau \leq \tau_0$ if $\vecu_0, \vecu_1,\dots,\vecu_{n-1}$ belong to the strip $\str$ and $t_{n}=n\tau\leq \tmax$, then
	\begin{equation*}
	\|\vecu(t_{n})-\vecu_{n}\|\leq C\tau |\log(\tau)| .
	\end{equation*}
	\item Since $C>0$ is a constant independent of $n$ and $\tau$, we can take $\tau_0>0$ sufficiently small such that for each $\tau \leq \tau_0$ we have $C\tau |\log(\tau)|\leq R$, the width of the strip $\str$, therefore by the previous step we have $\vecu_{n} \in \str$.
	\item Since $\vecu_0$ belongs to the strip and since $\tau_0$ and $C>0$ are independent of $n$, the proof can be concluded by induction.
\end{itemize}

Within the proof we will use the following conventions:
The positive constant $M$ comes from bounds for any of the analytic semigroups $T_0$, $S$, or $\calT$: For each $t \in (0,\tmax]$
\begin{equation}
\label{eq:smoothing}
\|T_0(t)\|,\|S(t)\|,\|\calT(t)\|\leq M, \qquad \text{and} \qquad \|t \, A_0 \, T_0(t)\|\leq M .
\end{equation}
Here the last estimate is usually referred to as the parabolic smoothing property of analytic semigroups, cf.~Remark~\ref{remark:analytic remarks} (c). 
By $C>0$ we will denote a constant that is independent of the time step, but may depend on other constants (e.g.~parameters of the problem) and on the exact solution (hence on the initial condition). Within a proof we shall indicate a possible increment of such appearing constants by a subscript: $\Co_1,\Co_2,\dots$, etc.

\begin{proof}[Proof of Theorem~\ref{theorem:Lie}]
	For the local Lipschitz continuity of the nonlinearity $\calF$, we will prove that the numerical solution remains in the strip $\Sigma$ around the exact solution $\vecu(t)$ using an induction argument.
	
	We estimate the global error $\vecu(t_n)-\vecu_n$, at time $t_n=n\tau\in(0,\tmax]$, by expressing it using the local error $\eloc_n = \vecu(t_n) - \bfL(\tau)(\vecu(t_{n-1}))$ as follows:
	\begin{align*}
	\vecu(t_n)-\vecu_n &= \vecu(t_n)-\bfL(\tau)\big(\vecu(t_{n-1})\big)+\bfL(\tau)\big(\vecu(t_{n-1})\big)-\bfL(\tau)\big(\vecu_{n-1}\big) \\
	&= \eloc_n+\calTT(\tau)\big(\vecu(t_{n-1})+\tau\calF(\vecu(t_{n-1}))\big)-\calTT(\tau)\big(\vecu_{n-1}+\tau\calF(\vecu_{n-1})\big) \\
	&= \eloc_n+\calTT(\tau)\big(\vecu(t_{n-1})-\vecu_{n-1}\big)+\tau\calTT(\tau)\varepsilon^{\calF}_{n-1} ,
	\end{align*}
	with the nonlinear difference term $\varepsilon^{\calF}_n = \calF(\vecu(t_n)) - \calF(\vecu_n)$. By resolving the recursion we obtain
	\begin{equation}
	\label{eq:global error representation}
	\begin{aligned}
	\vecu(t_n)-\vecu_n
	&= \eloc_n+\calTT(\tau)\big(\vecu(t_{n-1})-\vecu_{n-1}\big)+\tau\calTT(\tau)\varepsilon^{\calF}_{n-1}\\
	&= \eloc_n + \calTT(\tau)\eloc_{n-1} + \calTT(\tau)^2\big(\vecu(t_{n-2})-\vecu_{n-2}\big)  + \tau \calTT(\tau)^2 \varepsilon^{\calF}_{n-2} + \tau \calTT(\tau) \varepsilon^{\calF}_{n-1} \\
	&\hspace*{5.5pt}\vdots \\
	&= \eloc_n  + \sum_{j=1}^{n-1} \calTT(\tau)^j \eloc_{n-j} + \tau \sum_{j=1}^n \calTT(\tau)^j \varepsilon^{\calF}_{n-j} + \calTT(\tau)^n\big(\vecu(0)-\vecu_0\big).
	\end{aligned}
	\end{equation}
	Since we have $\vecu_0=\vecu(0)$, the last term vanishes.
	
	We now start the induction process. Let us assume that the error estimate \eqref{eq:Lie error estimate} holds for all $k \leq n - 1$ with $n \tau \leq \tmax$, i.e., for a $K>0$ independent of $\tau$ and $n$, we have
	\begin{equation}
	\label{eq:error estimates for the past}
	\text{for } k = 0, \dotsc, n - 1, \qquad \|\vecu(t_k)-\vecu_{k}\| \leq K \, \tau \, |\log(\tau)| .
	\end{equation}
	Below, we will show that the same error estimate also holds for $n$ as well.
	We note that, via $\vecu_0=\vecu(0)$, the assumed error estimate trivially holds for $n - 1 = 0$.
	

	\medskip
	We will now estimate the remaining terms of \eqref{eq:global error representation} in parts (i)--(iii), respectively. The estimates \eqref{eq:error estimates for the past} for the past values for $k$ only appear in part (iii).
	
	(i) We rewrite the local error $\eloc_n$ by using the forms \eqref{eq:VOC} and \eqref{eq:method1} of the exact and approximate solutions, respectively, and by Taylor's formula and (A\ref{item:5}) as 
	\begin{align}
	\nonumber \eloc_n&=\vecu(t_n)-\bfL(\tau)\big(\vecu(t_{n-1})\big) \\
	\nonumber &= \calT(\tau)\vecu(t_{n-1})+\int_0^\tau\hskip-1.0ex\calT(\tau-s)\calF(\vecu(t_{n-1}+s))\dd s - \calTT(\tau)\big(\vecu(t_{n-1})+\tau\calF(\vecu(t_{n-1}))\big) \\
	\nonumber &= \calT(\tau)\vecu(t_{n-1})+\int_0^\tau\calT(\tau-s)\calF(\vecu(t_{n-1}))\dd s \\
	\nonumber &\quad+ \int_0^\tau\calT(\tau-s)\int_0^s(\calF\circ\vecu)'(t_{n-1}+\xi)\dd\xi\dd s -\calTT(\tau)\big(\vecu(t_{n-1})+ \tau\calF(\vecu(t_{n-1}))\big) \\
	\nonumber &= \big(\calT(\tau)-\calTT(\tau)\big)\big(\vecu(t_{n-1})+\tau\calF(\vecu(t_{n-1}))\big) + \int_0^\tau\hskip-1.0ex\big(\calT(\tau-s)-\calT(\tau)\big) \calF(\vecu(t_{n-1}))\dd s \\
	\label{eq:three terms}&\quad+ \int_0^\tau\calT(\tau-s)\int_0^s(\calF\circ\vecu)'(t_{n-1}+\xi)\dd\xi\dd s.
	\end{align}
	In what follows we will estimate the three terms separately.
	
	\medskip
	We will bound the first term by using the boundedness of the semigroups  $T_0$ and $S$. Denote $(x,y)=\vecu(t_{n-1}) + \tau\calF(\vecu(t_{n-1}))$ and write 
	\begin{align*}
	& \big(\calT(\tau)-\calTT(\tau)\big)\vect xy\Bigr|_1 = \begin{pmatrix}0&Q_0(\tau)-V(\tau)\\0&0\end{pmatrix}\vect xy\Bigr|_1 \\
	=&\ Q_0(\tau)y-V(\tau)y = -\int_0^\tau T_0(\tau-\xi)D_0BS(\xi)y\dd\xi+\tau T_0(\tau)D_0BS(\tau)y.
	\end{align*}
	Whence we conclude
	\begin{equation*} 
	\big\|\bigl(\calT(\tau)-\calTT(\tau)\big)\vect xy\bigr\| \le\tau 2M^2\|D_0\|\|By\|\le \Co_1 \tau \|B(v(t_{n-1}) + \tau\calF_2(\vecu(t_{n-1})))\|.
	\end{equation*}
	The second term in \eqref{eq:three terms} can be estimated by Lemma \ref{lem:delta3}, and using  (A\ref{item:4}), as
	\begin{align*}
	\int_0^\tau \Big\| \big(\calT(\tau-s)-\calT(\tau)\big)\calF(\vecu(t_{n-1})) \Big\| \dd s \leq & \Co_2\tau\big(\|\calF(\vecu(t_{n-1}))\|+\|B\calF_2(\vecu(t_{n-1}) )\|\big).
	\end{align*}
	
	\noindent While, using the exponential boundedness of $\calT$ and (A\ref{item:5}), the third term in \eqref{eq:three terms} is directly bounded by
	\begin{align*}
	\int_0^\tau \int_0^s \Big\| \calT(\tau-s)(\calF\circ\vecu)'(t_{n-1}+\xi) \Big\| \dd\xi \dd s &\leq M\tau\| (\calF\circ\vecu)'\|_{\Ell^1([t_{n-1},t_n])}\\
	&\leq M\tau\| (\calF\circ\vecu)'\|_{\Ell^1([0,\tmax])}.
	\end{align*}
	%
	%
	Therefore, we finally obtain for the local error that
	\begin{equation}\label{eq:final estimate for (i)}
	\| \eloc_n \| \leq \Co_3 \tau.
	\end{equation}
	
	\medskip
	(ii) Since in each time step the local error is $\calO(\tau)$ and we have $\calO(1/\tau)$ time steps, a more careful analysis is needed for the the second term in \eqref{eq:global error representation}. We first rewrite this term by the variation of constants formula \eqref{eq:VOC} and the numerical method in the form \eqref{eq:method1}: 
	\begin{equation}
	\label{eq:red}
	\begin{aligned}
	&\ \sum_{j=1}^{n-1} \calTT(\tau)^j \eloc_{n-j} =\sum_{j=1}^{n-1} \calTT(\tau)^j \Big( \vecu(t_{n-j}) - \calTT(\tau) \big( \vecu(t_{n-j-1}) + \tau \calF(\vecu(t_{n-j-1})) \big) \Big) \\
	= &\ \sum_{j=1}^{n-1} \calTT(\tau)^j \big(\calT(\tau) - \calTT(\tau)\big) \vecu(t_{n-j-1}) \\
	&+ \sum_{j=1}^{n-1}\calTT(\tau)^j \, \Big( \! \int_0^\tau \! \calT(\tau-s) \calF(\vecu(t_{n-j-1}+s))\dd s - \tau \calTT(\tau) \calF(\vecu(t_{n-j-1})) \Big) . 
	\end{aligned}
	\end{equation}
	
	We rewrite the second term on the right-hand side of \eqref{eq:red} using Taylor's formula: 
	\begin{align*}
	& \calTT(\tau)^j \Big( \int_0^\tau \calT(\tau-s) \calF(\vecu(t_{n-j-1}+s))\dd s - \tau \calTT(\tau) \calF(\vecu(t_{n-j-1})) \Big) \\
	= &\ \calTT(\tau)^j\int_0^\tau \Big(\calT(\tau-s)\calF(\vecu(t_{n-j-1}+s))-\calTT(\tau)\calF(\vecu(t_{n-j-1}))\Big)\dd s \\
	= &\ \calTT(\tau)^j\Big(\int_0^\tau (\calT(\tau-s)-\calTT(\tau))\calF(\vecu(t_{n-j-1})) \dd s \\
	&\phantom{\calTT(\tau)^j\Big(} +\int_0^\tau\calT(\tau-s)\int_0^s(\calF\circ\vecu)'(t_{n-j-1}+\xi) \dd \xi  \dd s \Big) \\
	= &\ \int_0^\tau \hskip-1ex  \! \calTT(\tau)^j\big(\calT(\tau-s)-\calT(\tau)\big)\calF(\vecu(t_{n-j-1}))\dd s + \tau\calTT(\tau)^j\big(\calT(\tau)-\calTT(\tau)\big)\calF(\vecu(t_{n-j-1})) \\
	&\quad+ \int_0^\tau\int_0^s\calTT(\tau)^j\calT(\tau-s)(\calF\circ\vecu)'(t_{n-j-1}+\xi)\dd\xi\dd s.
	\end{align*}
	
	Combining the two identities above, for \eqref{eq:red} we obtain:
	\begin{equation}
	\label{eq:delta terms def}
	\begin{aligned}
	\sum_{j=1}^{n-1} \calTT(\tau)^j \eloc_{n-j} 
	&= \sum_{j=1}^{n-1} \big( \delta_{1,j} + \delta_{2,j} + \delta_{3,j} \big) \\
	\text{with} \qquad \delta_{1,j} &= \calTT(\tau)^j\big(\calT(\tau )- \calTT(\tau) \big) \Big(\vecu(t_{n-j-1}) + \tau\calF(\vecu(t_{n-j-1})) \Big) , \\
	\delta_{2,j} &= \int_0^\tau \calTT(\tau)^j\big(\calT(\tau-s)-\calT(\tau)\big)\calF(\vecu(t_{n-j-1})) \dd s , \\
	\delta_{3,j} &= \int_0^\tau \int_0^s \calTT(\tau)^j\calT(\tau-s)(\calF\circ\vecu)'(t_{n-j-1}+\xi) \dd\xi \dd s .
	\end{aligned}
	\end{equation}
	
	For the term $\delta_{1,j}$, upon setting $(x,y) = \vecu(t_{n-j-1}) + \tau\calF(\vecu(t_{n-j-1}))$ in Lemma \ref{lem:delta1}, by (A\ref{item:3}), (A\ref{item:4}), we obtain the following estimate for $j = 1,\dotsc,n-1$:
	\begin{align}
	\label{eq:delta1 final estimate}
	\|\delta_{1,j}\| \le \Co_4\tau^2 \|A_0 T_0(j\tau)\| \Big(& \big\| B \big( v(t_{n-j-1}) + \tau \calF_2(\vecu(t_{n-j-1})) \big) \big\| \\
	\nonumber	&+ \big\| B^2 \big( v(t_{n-j-1}) + \tau \calF_2(\vecu(t_{n-j-1})) \big) \big\| \Big) .
	\end{align}
	
	For the term $\delta_{2,j}$, setting $(x,y)=\calF(\vecu(t_{n-j-1}))$ in Lemma \ref{lem:delta2} by (A\ref{item:4}), we obtain the estimate for $j = 1,\dotsc,n-1$: 
	\begin{equation}
	\label{eq:delta2 final estimate}
	\begin{aligned}
	\|\delta_{2,j}\|&\leq \Co_5\tau^2 \|A_0 T_0(j\tau)\| \Big( \|\calF(\vecu(t_{n-j-1}))\|+\|B\calF_2(\vecu(t_{n-j-1}))\| \Big) \\
	&\qquad +\Co_6\tau^2 \Big( \|B\calF_2(\vecu(t_{n-j-1}))\|+\|B^2\calF_2(\vecu(t_{n-j-1}))\| \Big) .
	\end{aligned}
	\end{equation}
	
	The term $\delta_{3,j}$ is directly estimated by using Lemma~\ref{lem:calTTpower} and (A\ref{item:5}), for $j = 1,\dotsc,n-1$, as
	\begin{equation}
	\label{eq:delta3 final estimate}
	\begin{aligned}
	\|\delta_{3,j}\|& \leq \int_0^\tau \int_0^s  \Co_{7}(1+\log(j)) \Big\| \calT(\tau-s)(\calF\circ\vecu)'(t_{n-j-1}+\xi) \Big\| \dd \xi \dd s \\
	& \leq  M \Co_{7}(1+\log(j)) \int_0^\tau \int_0^s  \|(\calF\circ\vecu)'(t_{n-j-1}+\xi)\| \dd \xi \dd s \\
	& \leq \tau M \Co_{7}(1+\log(j)) \|(\calF\circ\vecu)'\|_{\Ell^1([t_{n-j-1},t_{n-j}])}. 
	\end{aligned}
	\end{equation}
	
	Finally, we combine the bounds \eqref{eq:delta1 final estimate}, \eqref{eq:delta2  final estimate}, \eqref{eq:delta3 final estimate}, respectively, for $\delta_{k,j}$, $k = 1,2,3$, then collecting the terms we obtain 
	\begin{equation}\label{eq:final estimate for (ii)}
	\begin{aligned}
	&\ \bigg\| \sum_{j=1}^{n-1} \calTT(\tau)^j \eloc_{n-j} \bigg\| \le \sum_{j=1}^{n-1} \Big( \|\delta_{1,j}\|+\|\delta_{2,j}\|+\|\delta_{3,j}\| \Big) \\
	\leq &\ \Co_8 \tau \sum_{j=1}^{n-1} \frac{1}{j} \, \Big( \|B v(t_{n-j-1})\| + \|B^2 v(t_{n-j-1})\| \Big)\\
	&\ +\Co_8 \tau  \sum_{j=1}^{n-1} \frac{1}{j} \, \Big(\|\calF(\vecu(t_{n-j-1}))\|  +\|B \calF_2(\vecu(t_{n-j-1})) \|\Big)\\
	&\ + \Co_9 \tau^2 \sum_{j=1}^{n-1} \Big( \|B\calF_2(\vecu(t_{n-j-1}))\|+\|B^2\calF_2(\vecu(t_{n-j-1}))\| \Big) \\
	&\ + \Co_{10}\tau\log(n)\|(\calF\circ u)'\|_{\Ell^1([0,\tmax])} \\
	\leq &\ \Co_{11} (1+\log(n))\tau + \Co_{12} \tau \, \leq \, \Co_{13} \tau\log(n+1)  ,
	\end{aligned}
	\end{equation}
	where we have used the parabolic smoothing property \eqref{eq:smoothing} of the analytic semigroup $T_0$ to estimate the factor by $\|A_0 T_0(j\tau)\| \leq M / (j\tau)$. 
	
	\medskip
	(iii) The errors in the nonlinear terms are estimated by using Lemma~\ref{lem:calTTpower} and the local Lipschitz continuity of $\calF$ in the appropriate spaces (see (A\ref{item:1}) and (A\ref{item:2})), in combination with the bounds \eqref{eq:error estimates for the past} for the past, as
	\begin{equation}\label{eq:final estimate for (iii)}
	\begin{aligned}
	&\ \bigg\| \tau \sum_{j=1}^n \calTT(\tau)^j \varepsilon^{\calF}_{n-j} \bigg\| 
	\leq\tau \sum_{j=1}^n \Big\| \calTT(\tau)^j \big( \calF(\vecu(t_{n-j})) - \calF(\vecu_{n-j}) \big) \Big\| \\
	\leq &\ \tau \sum_{j=1}^n M \big\| \calF(\vecu(t_{n-j})) - \calF(\vecu_{n-j}) \big\| 
	+ \tau \sum_{j=1}^n M \big\| B \big( \calF_2(\vecu(t_{n-j})) - \calF_2(\vecu_{n-j})\big) \big\| \\
	\leq &\ \tau \sum_{k=0}^{n-1} M   \LIP \| \vecu(t_{k}) - \vecu_{k} \| 
	+ \tau \sum_{k=0}^{n-1} M   \LIPB \| \vecu(t_{k}) - \vecu_{k} \|
	\leq \Co_{14}\tau \sum_{k=0}^{n-1} \| \vecu(t_{k}) - \vecu_{k} \| ,
	\end{aligned}
	\end{equation}
	recalling that $\LIP$ and $\LIPB$ are the Lipschitz constants on $\str$, see Assumptions~\ref{ass:solution} (A\ref{item:1}) and (A\ref{item:2}). 
	For the last inequality, we used here that $(\vecu_k)_{k=0}^{n-1}$ belongs to the strip $\str$ so that the Lipschitz continuity of $\calF$ can be used, see (A\ref{item:1}) and (A\ref{item:2}).
	
	The global error \eqref{eq:global error representation} is bounded by the combination of the estimates \eqref{eq:final estimate for (i)}, \eqref{eq:final estimate for (ii)}, and \eqref{eq:final estimate for (iii)} from (i)--(iii), which altogether yield
	\begin{equation}
	\label{eq:final estimate - pre Gronwall}
	\begin{aligned}
	\|\vecu(t_n) - \vecu_n\| 
	\leq &\ 
	\Co_3 \tau + \Co_{13}\log(n+1)\tau +\Co_{14} \tau \sum_{k=0}^{n-1} \| \vecu(t_{k}) - \vecu_{k} \| \\
	\leq &\ \Co_{15}\log(n+1)\tau+\Co_{14} \tau \sum_{k=0}^{n-1} \| \vecu(t_{k}) - \vecu_{k} \| .
	\end{aligned}
	\end{equation}
	A discrete Gronwall inequality then implies 
	\begin{equation}
	\label{eq:final estimate}
	\|\vecu(t_n) - \vecu_n\| \leq 
	\Co_{15} \ee^{\Co_{14} \tmax } \log(n+1) \tau \leq  C|\log(\tau)| \tau ,
	\end{equation}
	for $t_n=\tau n\in[0,\tmax]$, with the constant $C:=2\Co_{15}\ee^{\Co_{14}\tmax}>0$. Then for a $\tau_0>0$ sufficiently small such that for each $\tau \leq \tau_0$ we have $C|\log(\tau)| \tau\leq R$, then $\vecu_{n}\in \str$ and the error estimate \eqref{eq:Lie error estimate} is satisfied for $n$ as well. Hence \eqref{eq:error estimates for the past} holds even up to $n$ instead of $n-1$. Therefore, by induction, the proof of the theorem is complete.
\hfill \end{proof}

\begin{remark}\label{remark:finalafterproof}
	\begin{abc}
		\item Theorem~\ref{theorem:Lie} remains true, with an almost verbatim proof as above, if $B$ is merely assumed to be the generator of a $C_0$-semigroup.
		This requires the following additional condition:
		\vskip1ex
		
		\begin{numer}
			\item[(A5$'$)] The function $B\circ \calF_2\circ \vecu$ is differentiable and $(B\circ \calF_2\circ \vecu)'\in \Ell^1([0,\tmax];F)$. %
		\end{numer}
		\vskip1ex
		This is relevant only for the term $\delta_{3,j}$ in the inequality \eqref{eq:delta3 final estimate} when one applies the stability estimate from Lemma \ref{lem:calTTpower}.
		\item Time-dependent nonlinearities can also be allowed and the same error bound holds without essential modification of the previous proof. 
		Of course, the conditions (A\ref{item:1}), (A\ref{item:2}), (A\ref{item:4}), and (A\ref{item:5}) in Assumption~\ref{ass:solution}, involving $\calF$ and $\calF_2$ need to be suitably modified. For example the functions $\calF(t,\cdot)$ need  to be uniformly Lipschitz for $t\in [0,\tmax]$ (and even this can be relaxed a little), and the function $f$ defined by $f(t):=\calF(t,\vecu(t))$ needs to be differentiable, etc.
		
		\item The assumptions (A\ref{item:3}) and (A\ref{item:4}) involving the domain $\dom(B^2)$ may seem a little restrictive. However, in some applications these conditions are naturally satisfied: For example if $F$ is finite dimensional (such is the case for finite networks, see \cite{MugnoloRomanelli} or \cite{Sikolya_flows}). At the same time, these conditions seem to be optimal in this generality, and play a role only in the local error estimate of the Lie splitting, i.e.,  in Lemma \ref{lem:loclieA} and its applications. Indeed, at other places the conditions involving $\dom(B^2)$ are not needed.
		
		\item We would also like to point out that the assumptions (A\ref{item:3})--(A\ref{item:5}) are comparable to the assumptions required for \cite[Theorem~A.1]{dynbc_PDAE_splitting}, see, for instance, the second time derivatives in the last estimate of the proof therein. We also note that \cite[Appendix~A]{dynbc_PDAE_splitting} uses global Lipschitz assumptions.

	\end{abc}
\end{remark}

\section{Numerical experiments}
\label{section:numerics}

%
%

To illustrate and complement the error estimates of Theorem~\ref{theorem:Lie}, we have performed numerical experiments for Example~\ref{examp:Lip}.
\blueon In order to allow for a direct comparison with other works from the literature, we performed the same numerical experiments which were performed in \cite{dynbc} and \cite{dynbc_PDAE_splitting}. \blueoff 

Let $\Omega$ be the unit disk with boundary $\Ga = \{ x =(x_1,x_2)\in \R^2 \st \|x\|_2 = 1 \}$, with $\ga$ denoting the trace operator, and $\nnu$ denoting the outward unit normal field. Let us consider the boundary coupled  semilinear parabolic partial differential equation (PDE) system $u:\overline{\Om} \times [0,t_{\max}] \to \R$ and $v: [0,t_{\max}] \times  \Ga \to \R$ satisfying
\begin{equation}
\label{eq:PDE for numerical experiments}
\begin{cases}
\begin{alignedat}{3}
\pa_t u = &\ \Delta u + \calF_1(u,v) + \varrho_1 & \qquad & \text{ in } \Om , \\
\pa_t v = &\ \Delta_\Ga v + \calF_2(u,v) + \varrho_2 & \qquad & \text{ on } \Ga , \\
\ga u = &\ v & \qquad & \text{ on } \Ga , 
\end{alignedat}
\end{cases}	
\end{equation}
where the two nonlinearities are specified later, and where the two inhomogeneities $\varrho_1$ and $\varrho_2$ are chosen such that the exact solutions are known to be $u(x,t) = \exp(-t) x_1^2 x_2^2$ and $v(x,t) = \exp(-t) x_1^2 x_2^2$ (which naturally satisfy $\ga u = v$). The boundary coupled PDE system \eqref{eq:PDE for numerical experiments} fits into the abstract framework \eqref{eq:main0} in the sense of Example~\ref{examp:Lip}. We note that Theorem~\ref{theorem:Lie} still holds for \eqref{eq:PDE for numerical experiments} with the time-dependent inhomogeneities, see Remark~\ref{remark:finalafterproof}~(b). 

We performed numerical experiments using the splitting method \eqref{eq:method1}, written componentwise \eqref{eq:methodcomp min D_0}, which is applied to the bulk--surface finite element semi-discretisation, see \cite{ElliottRanner,dynbc}, of the weak form of \eqref{eq:PDE for numerical experiments}. 
The bulk--surface finite element semi-discretisation is based on a quasi-uniform triangulation $\Om_h$ of the continuous domain $\Om$, such that the discrete boundary $\Ga_h = \pa \Om_h$ is also a sufficient good approximation of $\Ga$. By this construction the traces of the finite element basis functions in $\Om_h$ naturally form a basis on the boundary $\Ga_h$, i.e.~$\{\ga_h \phi_j\}$ forms a boundary element basis on $\Ga_h$. 
Since $\Om \neq \Om_h$, the numerical and exact solutions are compared using the lift $^\ell$, given by $\eta_h^\ell(x^\ell) = \eta_h(x)$, where $x^\ell \in \Om$ is a suitable image of $x \in \Om_h$, see \cite[Section~4.1.1]{ElliottRanner}. 
For more details we refer to \cite[Section~4 and 5]{ElliottRanner}, or \cite[Section~3]{dynbc}.
Altogether this yields the matrix--vector formulation of the semi-discrete problem, for the nodal vectors $\bfu(t) \in \R^{N_\Om}$ and $\bfv(t) \in \R^{N_\Ga}$,
\begin{equation}
\label{eq:matrx-vector formulation}
\begin{cases}
\begin{aligned}
\bfM_\Om \dot \bfu + \bfA_\Om \bfu = &\ \bfF_1(\bfu,\bfv) + \bfrho_1 , \\
\bfM_\Ga \dot \bfv + \bfA_\Ga \bfv = &\ \bfF_2(\bfu,\bfv) + \bfrho_2 , \\
\bm{\ga} \bfu = &\ \bfv ,
\end{aligned}
\end{cases}	
\end{equation}
where $\bfM_\Om$ and $\bfA_\Om$ are the mass-lumped mass matrix and stiffness matrix for $\Om_h$, and similarly $\bfM_\Ga$ and $\bfA_\Ga$ for the discrete boundary $\Ga_h$, while the nonlinearities $\bfF_i$ and the inhomogeneities $\bfrho_i$ are defined accordingly. The discrete trace operator $\bm{\ga} \in \R^{N_\Ga \times N_\Om}$ extracts the nodal values at the boundary nodes. For all these quantities we have used quadratures of sufficiently high order such that the quadrature errors are negligible compared to all other spatial errors. For mass lumping in this context, and for its spatial approximation properties, we refer to \cite[Section~3.6]{dynbc}.

The two semigroups in 
\eqref{eq:methodcomp min D_0} are known, and are computed using the \verb|expmv| Matlab package of Al-Mohy and Higham \cite{expmv}, in the above matrix--vector formulation \eqref{eq:matrx-vector formulation} the (diagonal) mass matrices are transformed to the identity, i.e.~$\widetilde\bfA_\Om = \bfM_\Om^{-1} \bfA_\Om$, and similarly for $\widetilde\bfA_\Ga$, and all other terms. The numerical experiments were performed for this transformed system. In this setting the operator $D_0$ in \eqref{eq:D} corresponds to the harmonic extension operator, which we compute here by solving a Poisson problem with inhomogeneous Dirichlet boundary conditions.

\subsection{Convergence experiments}
We present two convergence experiments for the above boundary coupled PDE system, with nonlinearities:
\begin{subequations}
	\label{eq:nonlinearities}
	\begin{alignat}{3}
	\label{eq:nonlinearities - A-C}
	\calF_1(u,v) =&\  0 \qquad \text{and} & \qquad \calF_2(u,v) = &\ -v^3 + v, \\
	\label{eq:nonlinearities - mixing}
	\calF_1(u,v) =&\  u^2 \qquad \text{and} & \qquad \calF_2(u,v) = &\ v \ga u, 
	\end{alignat}
\end{subequations}
i.e.~with an Allen--Cahn-type nonlinearity, and a mixing nonlinearity.

This first nonlinearity was chosen since both \cite{dynbc} and \cite{dynbc_PDAE_splitting} have reported on numerical experiments for the same Allen--Cahn example $f_\Om(u) = 0$ and $f_\Ga(u) = -u^3 + u$, hence the  numerical experiments here are directly comparable to those referenced above. 
The second nonlinearity was chosen since it introduces some mixing between the bulk and boundary variables $u$ and $v$. 
We note here that the our experiments only serve as an illustration to Theorem~\ref{theorem:Lie}---and as comparison to the literature---since the nonlinearities in \eqref{eq:nonlinearities} do not satisfy (A\ref{item:1})--(A\ref{item:5}).

Using the splitting integrator \eqref{eq:method1}, in the form \eqref{eq:methodcomp min D_0}, we have solved the transformed system \eqref{eq:matrx-vector formulation} for a sequence of time steps $\tau_k = \tau_{k-1} / 2$ (with $\tau_0 = 0.2$) and a sequence of meshes with mesh width $h_k \approx h_{k-1} / \sqrt{2}$. 

In Figure~\ref{fig:convplot_Lie} and \ref{fig:convplot_Lie - A-C} we report on the $\Ell^\infty(\Ell^2(\Om))$ and $\Ell^\infty(\Ell^2(\Ga))$ errors of the two components for the two nonlinearities in \eqref{eq:nonlinearities}, comparing the (nodal interpolation of the) exact solutions and the numerical solutions. Here, analogously to the norm on $\Ell^\infty([0,T];X)$, for a sequence $(v_k)_{k=1}^N \subset X$ we set $\|v_k\|_{\Ell^\infty(X)} = \max_{k=1,\dotsc,N} \|v_k\|_X$.) In the log-log plots we can observe that the temporal convergence order matches the predicted convergence rate $\calO(\tau |\log(\tau)|)$ of Theorem~\ref{theorem:Lie}, note the dashed reference line  $\calO(\tau)$ (the factor $ |\log(\tau)|$ is naturally not observable).
In the figures each line (with different marker and colour) corresponds to a fixed mesh width $h$, while each marker on the lines corresponds to a time step size $\tau_k$.
The precise time steps and degrees of freedom values are reported in Figure~\ref{fig:convplot_Lie} and \ref{fig:convplot_Lie - A-C}.

We are reporting on the fully discrete errors of the numerical solution on multiple spatial grids, to be able to observe clearly that the convergence order is not reduced in contrast to \cite{dynbc}, cf.~\cite[Figure~2]{dynbc} where the error curves are sliding upward with each spatial refinement. Evidently, this is not the case here.


\begin{figure}[htbp]
	\includegraphics[width=\textwidth,clip,trim={70 5 70 5}] {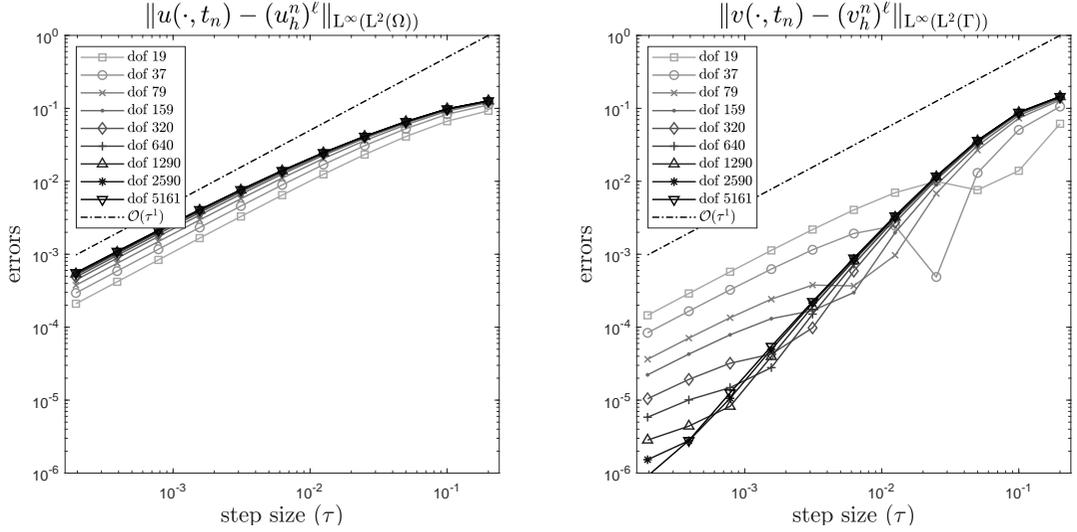}

	\caption{Temporal convergence plot for the splitting scheme 
		\eqref{eq:methodcomp min D_0} applied to the boundary coupled PDE system \eqref{eq:PDE for numerical experiments} with \eqref{eq:nonlinearities - mixing}, $\Ell^\infty(\Ell^2)$-norms of $u$ and $v$ components on the left- and right-hand sides, respectively.}
	\label{fig:convplot_Lie}
\end{figure}

\begin{figure}[htbp]
	\includegraphics[width=\textwidth]{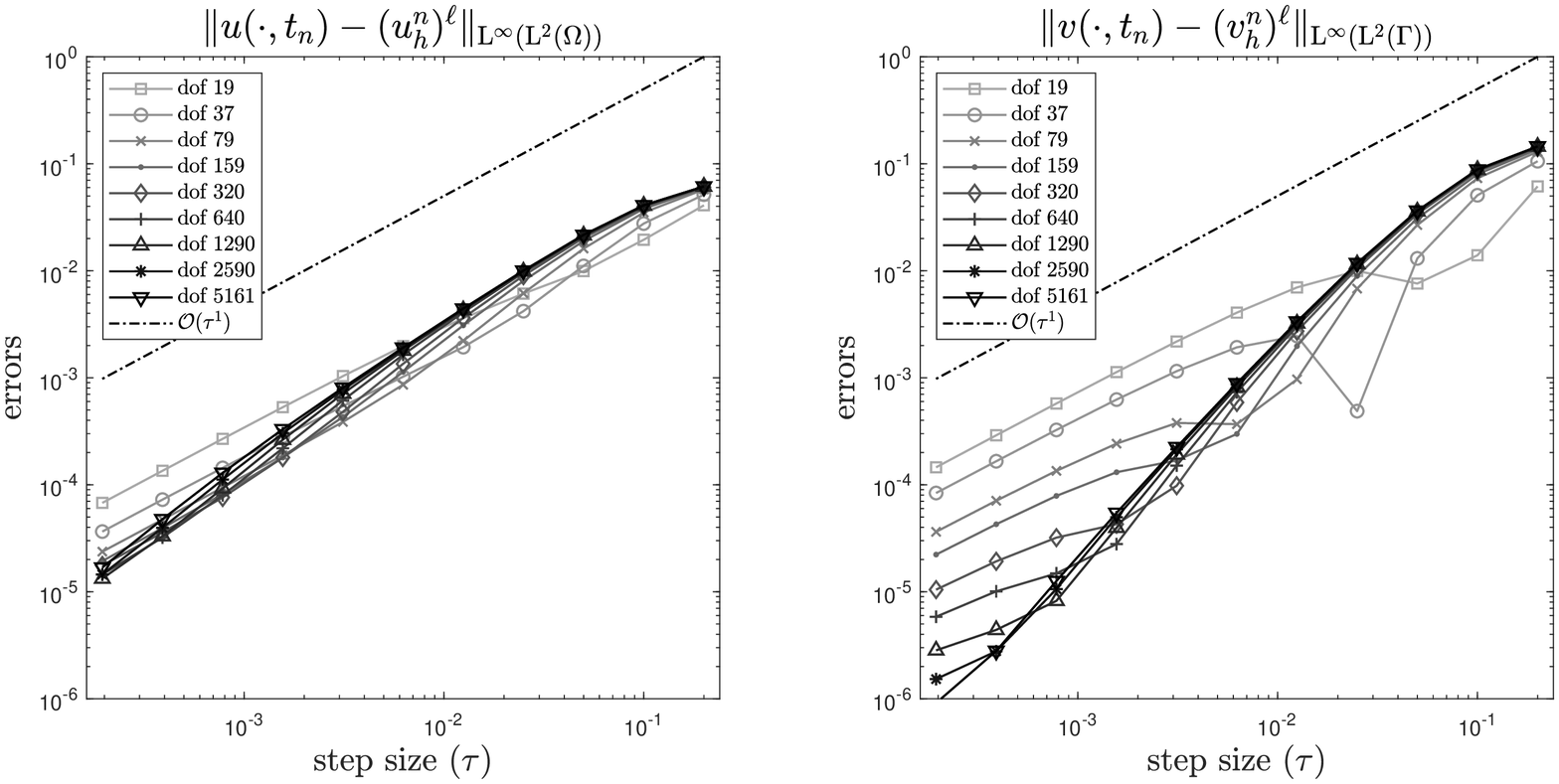}
	
	\caption{Temporal convergence plot for the splitting scheme 
			\eqref{eq:methodcomp min D_0} applied to the boundary coupled PDE system \eqref{eq:PDE for numerical experiments} with Allen--Cahn-type nonlinearity \eqref{eq:nonlinearities - A-C}, $\Ell^\infty(\Ell^2)$-norms of $u$ and $v$ components on the left- and right-hand sides, respectively.}
	\label{fig:convplot_Lie - A-C}
\end{figure}

\subsection{Convergence experiments with dynamic boundary conditions}
\label{section:numerics dynbc}

We performed the same convergence experiment for a partial differential equation with \emph{dynamic boundary conditions}, cf.~\cite{dynbc,dynbc_PDAE_splitting}, let $u:\overline{\Om} \times [0,\tmax] \to \R$ solve the problem
\begin{equation}
\label{eq:PDE dynbc for numerical experiments - classical form}
\begin{cases}
\begin{alignedat}{3}
\pa_t u = &\ \Delta u + f_\Om(u) + \varrho_1 & \qquad & \text{ in } \Om , \\
\pa_t u = &\ \Delta_\Ga u -\pa_\nnu u + f_\Ga(u) + \varrho_2 & \qquad & \text{ on } \Ga , 
\end{alignedat}
\end{cases}	
\end{equation}
using the same domain, exact solution, nonlinearities, etc.~as above.

Problem~\eqref{eq:PDE dynbc for numerical experiments - classical form} is equivalently rewritten as a boundary coupled PDE system \eqref{eq:PDE for numerical experiments},
where the two nonlinearities are given by
\begin{align*}
\calF_1(u,v) = f_\Om(u) \quad \text{and} \quad \calF_2(u,v) = -\pa_\nnu u + f_\Ga(u) .
\end{align*}
That is, the the nonlinear term $\calF_2$ incorporates the coupling through the Neumann trace $-\pa_\nnu u$. 
The numerical method \eqref{eq:method1}, written componentwise \eqref{eq:methodcomp min D_0}, is applied to this formulation with the nonlinearity $\calF_2$ containing the Neumann trace operator.

We again report on convergence test with the nonlinearities \eqref{eq:nonlinearities} in the roles of $f_\Om$ and $f_\Ga$.

In Figure~\ref{fig:convplot_Lie - dynbc} and \ref{fig:convplot_Lie - dynbc - A-C} we report on the $\Ell^\infty(\Ell^2(\Om))$ and $\Ell^\infty(\Ell^2(\Ga))$ error of the bulk and surface errors for the two nonlinearities in \eqref{eq:nonlinearities}, comparing the (nodal interpolation of the) exact solutions and the numerical solutions. (Figure~\ref{fig:convplot_Lie - dynbc} is obtained exactly as it was described for Figure~\ref{fig:convplot_Lie}, the precise time steps and degrees of freedom values can be read off from the figure.) 
Although in this case, due to the unboundedness of the Neumann trace operator in $\calF_2(u,v) = -\pa_\nnu u + f_\Ga(u)$, the conditions of Theorem~\ref{theorem:Lie} are not satisfied, in Figure~\ref{fig:convplot_Lie - dynbc} we still observe a convergence rate $\calO(\tau)$ (note the reference lines). 
Qualitatively we obtain the same plots for $\Ell^\infty(\He^1(\Om))$ and $\Ell^\infty(\He^1(\Ga))$ norms.

Note that our splitting method does not suffer from any type of order reduction, in contrast to the splitting schemes proposed in \cite{dynbc}, see Figure~1 and 2 therein. In \cite{dynbc_PDAE_splitting} the same order reduction issue was overcome by a different approach, using a correction term. 

\begin{figure}[htbp]
	\includegraphics[width=\textwidth,clip,trim={70 5 70 5}] {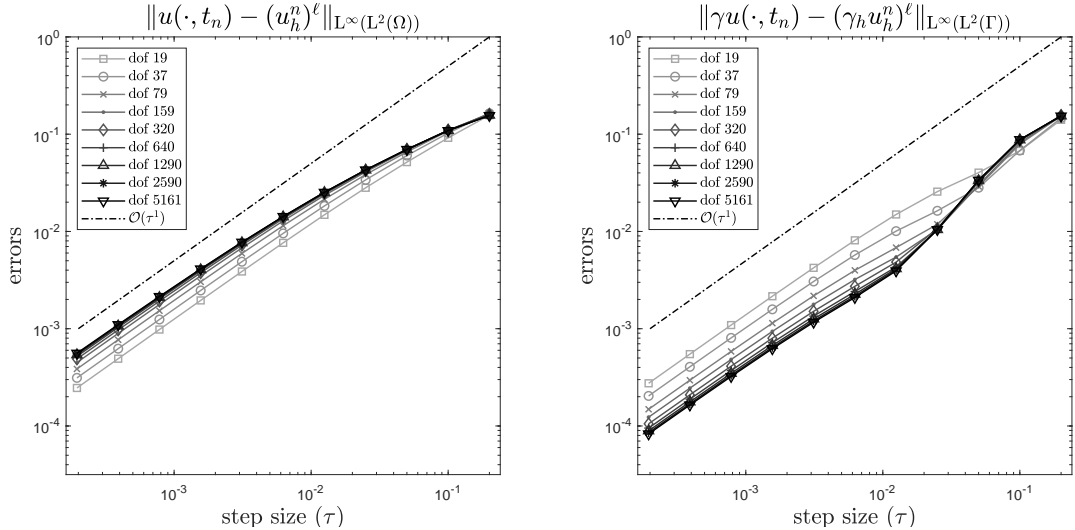}
	
	\caption{Temporal convergence plot for the splitting scheme 
		applied to the PDE with dynamic boundary conditions \eqref{eq:PDE dynbc for numerical experiments - classical form} with \eqref{eq:nonlinearities - mixing}, $\Ell^\infty(\Ell^2)$-norms of $u$ and $\ga u$ components on the left- and right-hand sides, respectively.}
	\label{fig:convplot_Lie - dynbc}
\end{figure}

\begin{figure}[htbp]
	\includegraphics[width=\textwidth] {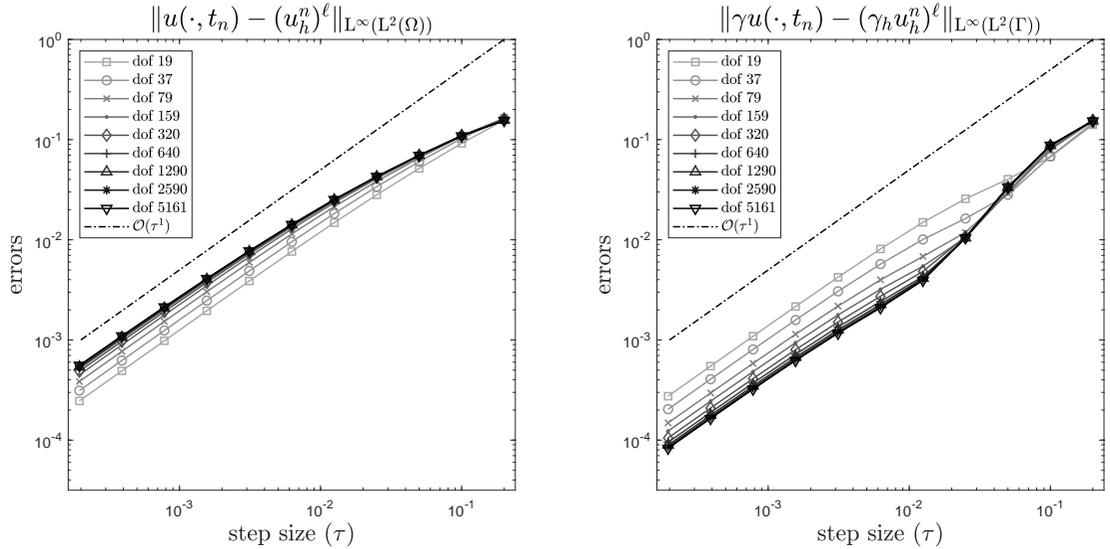}
	
	\caption{Temporal convergence plot for the splitting scheme 
			applied to the PDE with Allen--Cahn-type dynamic boundary condition \eqref{eq:PDE dynbc for numerical experiments - classical form} with \eqref{eq:nonlinearities - A-C}, $\Ell^\infty(\Ell^2)$-norms of $u$ and $\ga u$ components on the left- and right-hand sides, respectively.}
	\label{fig:convplot_Lie - dynbc - A-C}
\end{figure}

\section*{Acknowledgments}
The estimates within this paper were finalised during a Research in Pairs programme of all three authors at the Mathematisches Forschungsinstitut Oberwolfach (MFO). We are grateful for the Institute's support---and for its unparalleled inspiring atmosphere. 

\section*{Funding}
Petra Csom\'os acknowledges the Bolyai János Research Scholarship of the Hungarian Academy of Sciences. 

The work of Bal\'azs Kov\'acs is funded by the Heisenberg Programme of the Deut\-sche For\-schungs\-ge\-mein\-schaft (DFG, German Research Foundation) -- Project-ID 446\-431602, and by the DFG Graduiertenkolleg 2339 \emph{IntComSin} -- Project-ID 321821685.

The research was partially supported by the bilateral German-Hungarian Project \textit{CSITI -- Coupled Systems and Innovative Time Integrators}
financed by DAAD and Tempus Public Foundation.  This article is based upon work from COST Action CA18232 MAT-DYN-NET, supported by COST (European Cooperation in Science and Technology).


\bibliographystyle{abbrvnat}
\bibliography{ref-coupled-splitting-total}

\end{document}